\documentclass[draft]{amsart}

\usepackage{amsthm}
\usepackage{mathtools}
\usepackage{amssymb} 
\usepackage{enumerate} 

\usepackage{hyperref}
\usepackage{tikz-cd}
\numberwithin{equation}{section}

\theoremstyle{plain}
\newtheorem{lemma}{Lemma}[section]
\newtheorem{theorem}[lemma]{Theorem}
\newtheorem{proposition}[lemma]{Proposition}
\newtheorem{corollary}[lemma]{Corollary}

\theoremstyle{definition}

\theoremstyle{remark}
\newtheorem{remark}[lemma]{Remark} 

\renewcommand{\dim}{\operatorname{dim}}
\newcommand{\depth}{\operatorname{depth}}

\newcommand{\Ext}{\operatorname{Ext}}

\newcommand{\hh}{\operatorname{H}}

\newcommand{\ee}{{}^{e}\!}
\newcommand{\ef}{{}^{e+1}\!}

\newcommand{\vv}{\operatorname{v}}
\newcommand{\Hom}{\operatorname{Hom}}

\newcommand{\id}{\operatorname{id}}
\newcommand{\im}{\operatorname{im}}

\newcommand{\fd}{\operatorname{fd}}
\renewcommand{\le}{\leqslant}
\renewcommand{\ge}{\geqslant}

\newcommand{\Spec}{\operatorname{Spec}}
\newcommand{\Tor}{\operatorname{Tor}}
\newcommand{\pd}{\operatorname{pd}}

\newcommand{\RHom}{\operatorname{\mathsf{R}Hom}}

\newcommand{\lotimes}{\otimes^{\mathbf L}}

\newcommand{\sfD}{\mathsf D}

\newcommand{\fm}{\mathfrak{m}} 
\newcommand{\fp}{\mathfrak{p}}
\newcommand{\fn}{\mathfrak{n}} 
\newcommand{\fq}{\mathfrak{q}}

\begin{document}

\title[Frobenius and Cohen-Macaulay rings]{Frobenius and homological dimensions of complexes}

\author[T.\ Funk]{Taran Funk}

\address{University of Nebraska-Lincoln, Lincoln, NE 68588, U.S.A.}
\email{taran.funk@huskers.unl.edu}

\author[T.\ Marley]{Thomas Marley}

\address{University of Nebraska-Lincoln, Lincoln, NE 68588, U.S.A.}
\email{tmarley1@unl.edu}

\urladdr{http://www.math.unl.edu/~tmarley1}

\date{\today}

\bibliographystyle{amsplain}



\begin{abstract} 
It is proved that a module $M$ over a Noetherian local ring $R$ of prime characteristic and positive dimension  has finite flat dimension if  $\Tor_i^R(\ee R, M)=0$ for $\dim R$ consecutive positive values of $i$
and infinitely many $e$.  Here $\ee R$ denotes the ring $R$ viewed as an $R$-module via the $e$th iteration of the Frobenius endomorphism.  In the case $R$ is Cohen-Macualay, it suffices that the Tor vanishing above holds for a single $e\ge \log_p e(R)$, where $e(R)$ is the multiplicity of the ring.  This improves a result of D. Dailey, S. Iyengar, and the second author \cite{DIM}, as well as generalizing a theorem due to C. Miller \cite{Mi} from finitely generated modules to arbitrary modules.   We also show that if $R$ is a complete intersection ring then the vanishing of $\Tor_i^R(\ee R, M)$ for single positive values of $i$ and $e$ is sufficient to imply $M$ has finite flat dimension.  This extends a result of L. Avramov and C. Miller \cite{AM}.
 \end{abstract}

\maketitle

\section{Introduction}
For the past half-century the Frobenius endomorphism has proved to be an effective tool for characterizing when a given finitely generated module $M$ over a commutative Noetherian local ring $R$ of prime characteristic $p$ has certain homological properties.   In 1973 Peskine and Szpiro \cite{PS} proved that if a finitely generated module $M$ has finite projective dimension then $\Tor_i^R(\ee R, M)=0$ for all positive integers $i$ and $e$, where $\ee R$ denotes the ring $R$ viewed as an $R$-module via the $e$th iteration of the Frobenius endomorphism.    Shortly thereafter, Herzog \cite{He} proved the converse:  In fact, he showed that if  $\Tor_i^R(\ee R, M)=0$ for all $i>0$ and infinitely many $e$ then $M$ has finite projective dimension.    Some twenty years later, Koh and Lee \cite{KL} established a stronger version of Herzog's result:  If for some $e$ sufficiently large $\Tor_i^R(\ee R, M)=0$ for $\depth R +1$ consecutive positive values of $i$ then $M$ has finite projective dimension.    Later,  Miller \cite{Mi} showed that if $R$ is CM of positive dimension then $\dim R$ consecutive vanishings of $\Tor_i^R(\ee R, M)$  for some $e$ sufficiently large implies $M$ has finite projective dimension.   

One may ask to what extent do the above results hold for arbitrary (i.e., not necessarily finitely generated) modules.  As $\Tor$ detects flatness rather than projectivity,  we seek conditions which imply a given module has finite flat dimension.  (It is a deep result of Jensen \cite[Proposition 6]{J} and Raynaud and Gruson \cite[Seconde partie, Th\'{e}or\`{e}me 3.2.6]{RG} that, in the case $R$ has finite Krull dimension, a module $M$ has finite flat dimension if and only if it has finite projective dimension.  We choose not to make use of this result in this paper, however.)   In \cite{MW}, the second author together with M. Webb proved the analogue of Peskine and Szpiro's result for modules of finite flat dimension;  that is, if $M$ has finite flat dimension then $\Tor_i^R(\ee R, M)=0$ for all positive integers $i$ and $e$.  Further, it was shown 
that that the analogue of Herzog's result holds for arbitrary modules as well.    Subsequently, Dailey, Iyengar and the second author \cite{DIM} showed that if $\Tor_i^R(\ee R, M)=0$ for $\dim R+1$ consecutive positive values of $i$ and infinitely many $e$, then $M$ has finite flat dimension.  The question remains whether, as in the results of Koh, Lee, and Miller, one can get by with fewer consecutive vanishings for arbitrary modules when $\dim R>0$.


In the present paper,  we show that in fact  $\dim R$ consecutive vanishings of $\Tor_i^R(\ee R, M)$ for positive values of $i$ and infinitely many $e$ is sufficient to prove that $M$ has finite flat dimension if $\dim R>0$; if  $R$ is Cohen-Macaulay, it suffices to show these vanishings hold for some $e$ greater than the multiplicity of the ring.   

We also prove in the case $R$ is a local complete intersection ring that the vanishing of $\Tor_i^R(\ee R, M)$ for some positive integers $i$ and $e$ imply that $M$ has finite flat dimension.    This generalizes a result of Avramov and Miller \cite{AM}, who established this for finitely generated modules.   We also show that all of the above results are valid for complexes.   The following theorem summarizes our main results:

\begin{theorem}  \label{main-theorem} Let $(R,\fm)$ be a local ring of prime characteristic $p$ and positive dimension $d$.  Let $M$ be an $R$-complex such that $\hh_*(M)$ is bounded above.  The following are equivalent:
\begin{enumerate}[(a)]
\item $M$ has finite flat dimension;
\item There exists $t>\sup \hh_*(M)$ such that $\Tor_i^R(\ee R, M)=0$ for $t\le i\le t+d-1$ and infinitely many $e$.
\end{enumerate}
If $R$ is Cohen-Macaulay, condition (a) is equivalent to:
\begin{enumerate}[(c)]
\item There exists $t>\sup \hh_*(M)$ such that $\Tor_i^R(\ee R, M)=0$ for $t\le i\le t+d-1$ for some $e\ge \log_p e(R)$, where $e(R)$ denotes the multiplicity of $R$.
\end{enumerate}
If $R$ is a local complete intersection of arbitrary dimension, then condition (a) is equivalent to:
\begin{enumerate}[(d)]
\item $\Tor_i^R(\ee R, M)=0$ for some $i>\sup \hh_*(M)$ and some $e>0$.
\end{enumerate}
\end{theorem}

Analogous results hold for $\Ext^i_R(\ee R, M)$ and injective dimension in the case the Frobenius endomorphism is finite.  In fact, with the exception of the proof of (d) implies (a), our method of proof
is to first establish the results for injective dimension and then deduce the corresponding statements for flat dimension using standard arguments.

\medskip
{\it Acknowledgment:}  We thank Olgur Celikbas and Yongwei Yao for pointing out an error in the proof of Theorem 3.1 in an earlier version of this paper.

\section{Preliminaries}

Throughout this paper $(R, \fm, k)$ will denote a commutative Noetherian local ring with maximal ideal $\fm$ and residue field $k$.  In the case $R$ has prime characteristic $p$, we let $f:R\to R$
denote the Frobenius endomorphism; i.e., $f(r)=r^p$ for every $r\in R$.   For an integer $e\ge 1$ we let $\ee R$ denote the ring $R$ viewed as an $R$-algebra via $f^e$; i.e., for $r\in R $ and $s\in \ee R$, $r\cdot s:= f^e(r)s=r^{p^e}s$.   If $\ee R$ is finitely generated as an $R$-module for some (equivalently, all) $e>0$, we say that $R$ is {\it $F$-finite}.

We refer the reader to \cite{AF1991} for terminology and conventions regarding complexes.   If $M$ is an $R$-complex, we write $M_*$ (respectively, $M^*$) to emphasize when we are indexing $M$ homologically  (respectively, cohomologically).
It will occasionally be useful to work in the derived category of $R$, which will be denoted by $\sfD (R)$.  We use the symbol `$\simeq$' to denote  an isomorphism in $\sfD (R)$.  

We first establish how the $R$-algebra $\ee R$ (i.e., restriction of scalars) behaves with respect to flat extensions.  Much of this is folklore, but we include it for the reader's convenience.

\begin{lemma} \label{flatsquare} Consider a commutative square of ring homomorphisms:
$$
\begin{tikzcd} 
A \arrow {r} \arrow{d} & B \arrow{d} \\
C\arrow{r} & D
\end{tikzcd}
$$
where $B$ is flat over $A$, and $D$ is flat over $C$.
Then for any $A$-complex $M$ and any $C$-complex $N$ one has for each $i$ an isomorphism of $D$-modules
$$\Tor_i^A(M,N)\otimes_C D\cong \Tor_i^B(M\otimes_A B, N\otimes_C D).$$ 
\end{lemma}
\begin{proof} We have the following isomorphisms in $\sfD(D)$:
$$
\begin{aligned} (M\lotimes_A N)\otimes_C D&\simeq (M\lotimes_A D)\lotimes_C N\\
&\simeq M\lotimes_A (B\lotimes_B D)\lotimes_C N \\
&\simeq (M\otimes_A B)\lotimes_B (D\otimes_C N).
\end{aligned}
$$
Taking homology and using that $-\otimes_C D$ is exact gives the desired result.
\end{proof}

\begin{corollary} \label{Frob-flat-ext} Suppose $R$ has prime characteristic and $S$ is a flat $R$-algebra.  Let $M$ be an $R$-complex and $e$ a positive integer.  Then for each $i$ there is an isomorphism of $\ee S$-modules
$$\Tor_i^R(M,\ee R)\otimes_{\ee R} \ee S \cong \Tor_i^S(M\otimes_R S, \ee S).$$
\end{corollary}
\begin{proof} We have a commutative square of ring maps:
$$
\begin{tikzcd} 
R \arrow {r} \arrow{d} & S \arrow{d} \\
\ee R\arrow{r} & \ee S
\end{tikzcd}
$$
Since $S$ is flat over $R$, $\ee S$ is flat over $\ee R$.  The result now follows from Lemma \ref{flatsquare}.
\end{proof}

\begin{lemma} \label{infinite-residue}
Let $(R,\fm)$ be a local ring of prime characteristic $p$ which is F-finite. Let $x$ be an indeterminate over $R$, $S:=R[x]_{\fm R[x]}$, and $T:=(\ee R)[x]_{\fn (\ee R)[x]}$, where $\fn$ is the maximal ideal of $\ee R$.   Then
\begin{enumerate}[(a)]
\item $\ee S$ is a free $T$-module of rank $p^e$.
\item $T$ is a finitely generated $S$-module.
\item $S$ is F-finite.
\item For each $R$-complex $M$ with $\hh^*(M)$ bounded below and for each $i$, there is an isomorphism of $\ee S$-modules
$$\Ext^i_S(\ee S, M\otimes_R S)\cong \Hom_T(\ee S, \Ext^i_R(\ee R, M)\otimes_R S).$$
\end{enumerate}
\end{lemma}
\begin{proof}
Let $A=R[x]$, $B=(\ee R)[x]$, and $C=(\ee R)[x^{\frac{1}{p^e}}]\cong \ee A$.    Note that $C$ is a free $B$-module of rank $p^e$ and $B$ is a f.g. $A$-module.
Let $U=A\setminus \fm A$, $V=B\setminus \fn B$, and $W=C\setminus \fn C$.   Then $A_U=S$.  It is straightforward to check that $T=B_V=B_U$ and $\ee S=C_W=C_V$.   Hence, (a), (b), and (c) are immediate.

We have the following isomorphisms of $\ee S$-modules.  
$$
\begin{aligned} \Hom_T(\ee S, \Ext^i_R(\ee R, M)\otimes_R S)&\cong \Hom_T(\ee S, \Ext^i_S(T, M\otimes_R S))\\
&\cong \Hom_T(\ee S, \hh^i(\RHom_S (T, M\otimes_R S)))\\
&\cong \hh^i(\RHom_T(\ee S, \RHom_S(T, M\otimes_R S)))\\
&\cong \hh^i(\RHom_S(\ee S, M\otimes_R S)) \\
&\cong \Ext^i_S(\ee S, M\otimes_R S).
\end{aligned}
$$
The first isomorphism follows since $S$ is flat over $R$, $\ee R$ is finitely generated over $R$, and $\hh^*(M)$ is bounded below (see \cite[Lemma 4.4(F)]{AF1991}). The third isomorphism holds as $\ee S$ is a free $T$-module.

\end{proof}

\begin{corollary} \label{infinite-residue-cor} With the notation as in part (d) of Lemma \ref{infinite-residue},  for each $i$ we have that $\Ext^i_R(\ee R, M)=0$ if and only if $\Ext^i_S(\ee S, M\otimes_R S)=0$.
\end{corollary}
\begin{proof} Since $\Hom_T(\ee S, -)$ and $-\otimes_R S$ are faithful functors, the result follows from part (4) of Lemma \ref{infinite-residue}.
\end{proof}

The following result is also well-known:

\begin{lemma} \label{duality-isos} Let $R$ be a commutative Noetherian ring, $M$, $N$ $R$-complexes, and $I$ an injective $R$-module.
\begin{enumerate} [(a)]
\item For all $i$ we have isomorphisms
$$\Hom_R(\Tor_i^R(M,N),I)\cong \Ext^i_R(M, \Hom_R(N,I)).$$
\item Suppose $\hh_*(M)$ is bounded below, $H_i(M)$ is finitely generated for all $i$, and $\hh^*(N)$ is bounded below.  Then for all $i$ we have isomorphisms
$$\Tor_i^R(M, \Hom_R(N,I))\cong \Hom_R(\Ext^i_R(M,N), I).$$
\end{enumerate}
\end{lemma}
\begin{proof} Using adjunction and \cite[Lemma 4.4(I)]{AF1991}, we have the following isomorphisms in $\sfD(R)$:
$$
\begin{aligned}
\Hom_R(M\lotimes_R N, I)&\simeq \RHom_R(M\lotimes_R N, I) \\
&\simeq \RHom_R(M, \RHom_R(N, I))\\
&\simeq \RHom_R(M, \Hom_R(N,I))
\end{aligned}
$$
and
$$
\begin{aligned}
M\lotimes_R \Hom_R(N,I)\simeq \Hom_R(\RHom_R(M,N), I).
\end{aligned}
$$
Taking homology and using that $\Hom_R(-,I)$ is an exact functor yields the desired isomorphisms.

\end{proof}

For an $R$-complex M, let $M^{\sharp}$ denote the complex which has the same underlying graded module as $M$ and whose differentials are all zero.
Let  $\fd_R M$ denote the flat dimension of $M$; that is, $$\fd_R M=\inf \{\sup \hh_*(F^{\sharp}) \mid F\simeq M, \ F \text{ semi-flat}\}.$$
Similarly, $\id_R M$ will denote the injective dimension of $M$, i.e.,
$$\id_R M=\inf \{\sup \hh^*(I^{\sharp}) \mid I\simeq M, \ I \text{ semi-injective}\}.$$

\begin{corollary} \label{id-fd} Let $(R,\fm)$ be a local ring,  $E=E_R(R/\fm)$, and let $(-)^{\vv}$ denote the functor $\Hom_R(-,E)$.  Let $M$ be an $R$-complex.  Then
\begin{enumerate}[(a)]
\item $\fd_R M\le\id_R M^{\vv}$ with equality if $\hh_*(M)$ is bounded below.
\item If $\hh^*(M)$ is bounded below, then $\id_R M= \fd_R M^{\vv}$
\end{enumerate}
\begin{proof} Using \cite[Proposition 5.3.F]{AF1991} and Lemma \ref{duality-isos} with $I=E$, we have:
$$
\begin{aligned} 
\fd_R M&=\sup  \{ j \mid \Tor^R_j(R/\fp,M)\neq 0 \text{ for some }\fp\in \Spec R\} \\
&= \sup \{ j \mid \Tor^R_j(R/\fp,M)^{\vv}\neq 0 \text{ for some }\fp\in \Spec R\} \\
&=\sup \{j \mid \Ext^j_R(R/\fp,M^{\vv})\neq 0 \text{ for some }\fp\in \Spec R\} \\
&\le \id_R M^{\vv},
\end{aligned}
$$
where equality holds in the last line if $\hh^*(M^{\vv})$ is bounded below, or equivalently, if $\hh_*(M)$ is bounded below.
Part (b) is proved similarly.

\end{proof}
\end{corollary}

We note the following remark, which will be needed in the subsequent sections:

\begin{remark} \label{id-flat}  Let $S$ be a faithfully flat $R$-algebra and $M$ an $R$-complex.  Then
\begin{enumerate}[(a)]  
\item $\fd_R M= \fd_S M\otimes_R S$;
\item  If $\hh^*(M)$ is bounded below, then $\id_R M\le \id_S M\otimes_R S$.
\end{enumerate}
\end{remark}
\begin{proof}
For part (a), note that $\fd_R M\ge \fd_S M\otimes_R S$, since $-\otimes_R S$ preserves quasi-isomorphisms and $F\otimes_R S$ is a semi-flat $S$-complex whenever $F$ is a semi-flat $R$-complex.   For the reverse inequality, we have by \cite[Propositon 5.3.F]{AF1991}, 
$$
\begin{aligned}
\fd_R M&=\sup \{ j \mid \Tor^R_j(R/\fp,M)\neq 0 \text{ for some }\fp\in \Spec R\} \\
&=\sup \{ j \mid \Tor_j^R(R/\fp,M)\otimes_R S\neq 0 \text{ for some }\fp\in \Spec R\} \\
&=\sup \{ j \mid \Tor^S_j(S/\fp S,M\otimes_R S)\neq 0 \text{ for some }\fp\in \Spec R\} \\
&\le \fd_S M\otimes_R S.
\end{aligned}
$$
For part (b), we have by \cite[Proposition 5.3.I]{AF1991} that

$$
\begin{aligned}
\id_R M&=\sup \{ j \mid \Ext^j_R(R/\fp,M)\neq 0 \text{ for some }\fp\in \Spec R\} \\
&=\sup \{ j \mid \Ext^j_R(R/\fp,M)\otimes_R S\neq 0 \text{ for some }\fp\in \Spec R\} \\
&=\sup \{ j \mid \Ext^j_S(S/\fp S,M\otimes_R S)\neq 0 \text{ for some }\fp\in \Spec R\} \\
&\le \id_S M\otimes_R S.
\end{aligned}
$$
\end{proof}

Finally, we will need the following result for zero-dimensional rings.  It is a special case of Theorem 1.1 of \cite{DIM} (or more properly, its dual), but as the proof is short, we include it here for the reader's convenience:

\begin{proposition} \label{zero-dim} Let $(R,\fm,k)$ be a zero-dimensional local ring of prime characteristic $p$.  Let $M$ be an $R$-module and $e\ge \log_p \lambda(R)$ an integer, where $\lambda(-)$ denotes length.  If $\Ext^i_R(\ee R, M)=0$ for some $i>0$ then $M$ is injective.
\end{proposition}
\begin{proof}  By \cite[Proposition 4.1 and Corollary 5.3]{B}, if $M$ has finite injective dimension then $\id_R M\le \dim R$.  Hence, it suffices to show $\id_R M<\infty$.   By replacing $M$ with a syzygy of an injective resolution of $M$, we may assume $\Ext^1_R(\ee R, M)=0$.  Since $p^e\ge \lambda(R)$,
we have $\fm^{p^e}=0$.   Then $\fm\cdot \ee R=0$ and thus $\ee R$ is a $k$-vector space.  Hence, $\ee R\cong k^{\ell}$ as $R$-modules, for some (possibly infinite) $\ell>0$.   Thus, the condition $\Ext^1_R(\ee R,M)=0$ implies $\Ext^1_R(k,M)=0$.  Hence, $M$ is injective.
\end{proof}

\section{The Cohen-Macaulay case}

For an $R$-complex $M$ and $\fp\in \Spec R$, we let $\mu_i(\fp, M):=\dim_{k(\fp)} \Ext^i_{R_{\fp}}(k(\fp), M_{\fp})$.  If $\hh^*(M)$ is bounded below, $\mu_i(\fp, M)$ is the number (possibly infinite) of copies of $E_R(R/\fp)$ in $I^i$, where $I$ is a minimal semi-injective resolution of $M$.

\begin{theorem} \label{inj-dim} Let $(R, \fm, k)$ be a $d$-dimensional Cohen-Macaulay local ring of prime characteristic $p$ and which is $F$-finite. Let $e\ge \log_p e(R)$ be an integer, $M$ an $R$-complex, and $r=\max\{1,d\}$.
\begin{enumerate}[(a)]
\item Suppose there exists an integer $t> \sup \hh^*(M)$ such that $\Ext^i_R(\ee R, M)=0$ for $t\le i\le t+r-1$.  Then $M$ has finite injective dimension.
\item Suppose there exists an integer $t>\sup \hh_*(M)$ such that $\Tor_i^R(\ee R, M)=0$ for $t\le i\le t+r-1$. Then $M$ has finite flat dimension.
\end{enumerate}
\end{theorem}
\begin{proof}  We first note that if (a) holds in the case $\dim R=d$, then (b) also holds in the case $\dim R=d$:  For, suppose the hypotheses of (b) hold for a complex $M$.  Then by Lemma \ref{duality-isos}(a), $\Ext^i_R(\ee R, M^{\vv})\cong \Tor_i^R(\ee R, M)^{\vv}=0$ for $t\le i\le t+r-1$.   As $\sup \hh^*(M^{\vv})=\sup \hh_*(M)$, we have by (a) that $\id_R M^{\vv}<\infty$.   Hence, $\fd_R M<\infty$ by Lemma \ref{id-fd}(a).

Thus, it suffices to prove (a).
If $\id_R M<t-1$ there is nothing to prove.  Otherwise, let  $J$ be a minimal semi-injective resolution of $M$ and $Z:=Z^{t-1}(J)$ be the (necessarily nonzero) subcomplex consisting of the cycles  of degree $t-1$ of $J$.  As $t-1\ge \sup \hh^*(M)$,  $J^{\ge t-1}$ is a minimal semi-injective resolution of 
$Z$ and $\id_R M=\id_R Z$.   Furthermore, from the exact sequence of complexes 
$$0\to J^{\ge t-1} \to J \to J^{<t-1}\to 0$$
we have that $\Ext^i_R(\ee R, Z)\cong \Ext^i_R(\ee R, M)$ for all $i\ge t$.  Hence, without loss of generality, we may assume (after shifting) that $M$ is a module concentrated in degree zero and $\Ext^i_R(\ee R,M)=0$ for $i=1,\dots,r$.
Also, by replacing $R$ with $R[x]_{\fm R[x]}$, if necessary, we may assume $R$ has an infinite residue field (Lemma \ref{infinite-residue} and Remark \ref{id-flat}).

We proceed by induction on $d$, with the case $d=0$ being established by Proposition \ref{zero-dim}.   Suppose $d\ge 1$ (so $r=d$) and we assume both (a) and (b) hold for complexes over local rings of dimension less than $d$.  Let
 $\fp\neq \fm$ be a prime ideal of $R$.  As $R$ is $F$-finite, we have $\Ext^i_{R_{\fp}}(\ee R_{\fp}, M_{\fp})=0$
for $1\le i\le d$.  As $d\ge \max\{1, \dim R_{\fp}\}$ and $e(R)\ge e(R_{\fp})$ (see \cite{L}), we have $\id_{R_\fp} M_{\fp}<\infty$ by the induction hypothesis.
Hence, $\id_{R_\fp} M_{\fp}\le \dim R_{\fp}\le d-1$ by \cite[Proposition 4.1 and Corollary 5.3]{B}.   It follows that $\mu_i(\fp, M)=0$ for all $i\ge d$ and all $\fp\neq \fm$.

For convenience, we let $S$ denote the $R$-algebra $\ee R$ and $\fn$ the maximal ideal of $S$. As $S/\fn$ is infinite, we may choose a system of parameters ${\bf x}=x_1,\dots,x_d\in \fn$ such that $({\bf x})$ is a minimal reduction of $\fn$. 
Then $\lambda_S(S/({\bf x}))=e(S)=e(R)$ and $\fm \cdot S/({\bf x})=\fn^{[p^e]} S/({\bf x})=0$, as $p^e\ge \lambda_S (S/({\bf x}))$.

As $J$ is a minimal injective resolution of $M$, we have by assumption that
\begin{equation}
\label{eq:exact1}
\Hom_R(S, J^0)\xrightarrow{\phi^0} \Hom_R(S, J^1)\to \cdots \to \Hom_R(S, J^{d}) \xrightarrow{\phi^{d}} \Hom_R(S, J^{d+1})
\end{equation}
is exact.  Let $L$ be the injective $S$-envelope of $\operatorname{coker}{\phi^{d}}$ and $\psi:\Hom_R(S, J^{d+1})\to L$ the induced map.
Hence, 
$$0\to \Hom_R(S, J^0)\to \cdots \xrightarrow{\phi^d} \Hom_R(S, J^{d+1}) \xrightarrow{\psi} L$$
is acyclic and in fact the start of an injective $S$-resolution of $\Hom_R(S, M)$.  Setting $\overline{S}=S/({\bf x})$ and applying $\Hom_S(\overline{S},-)$ to the above resolution yields
an exact sequence
\begin{equation} \label{eq:exact2} \Hom_S(\overline{S}, \Hom_R (S, J^d))\xrightarrow{\overline{\phi^d}} \Hom_S(\overline{S}, \Hom_R(S, J^{d+1}))\xrightarrow{\overline{\psi}} \Hom_S (\overline{S}, L).
\end{equation}
The exactness holds as $\pd_S \overline{S}=d$ and thus $\Ext^{d+1}_S(\overline{S}, \Hom_S(S,M))=0$.  

Since $\overline{S}$ is a finitely generated $R$-module and annihilated by $\fm$, we have $\overline{S}\cong k^t$ as $R$-modules for some $t$.  Thus, the exact sequence (\ref{eq:exact2}) is naturally isomorphic to
$$
\Hom_R(k^t, J^d)\xrightarrow{\overline{\phi^d}} \Hom_R(k^t, J^{d+1})\xrightarrow{\overline{\psi}}\Hom_S (\overline{S}, L).
$$
As $J$ is minimal, we have $\overline{\phi^d}$ is the zero map and hence $\overline{\psi}$ is injective.
\smallskip

{\it Claim:}  $\psi$ is injective.

{\it Proof:}  Let $K=\ker \psi$.   Applying $\Hom_S(\overline{S}, -)$ to  
$$0\to K\to \Hom_R(S, J^{d+1})\xrightarrow{\psi} L$$
we see that $\Hom_S(\overline{S}, K)=0$.  Since $\mu_{d+1}(\fp, M)=0$ for all primes $\fp\neq \fm$, we obtain that $J^{d+1}=\oplus_{\alpha\in I} E_R(k)$ for some (possibly infinite) index set $I$.  Since $S$ is a finite $R$-module, we have $\Hom_R(S, J^{d+1})_{\fp}\cong \Hom_{R_{\fp}}(S_{\fp}, J^{d+1}_{\fp})=0$ for all $\fp\neq \fm$.  Hence,  $\Hom_R(S, J^{d+1})_{\fq}=0$ for all $\fq\in \Spec S$, $\fq\neq \fn$.   Thus $\Hom_R(S, J^{d+1})$, and consequently $K$, is $\fn$-torsion.   Thus, if $K\neq 0$,  we must have $\Hom_S(\overline{S}, K)\neq 0$.  We conclude $K=0$ and $\psi$ is injective.

\smallskip

Now consider the complex $J$, which is a minimal injective resolution of $M$:

$$0\to J^0\xrightarrow{\partial^0} J^1\to \cdots\to J^{d-1} \xrightarrow{\partial^{d-1}} J^d\xrightarrow{\partial^d}\cdots$$

The proof will be complete upon proving:
\smallskip

{\it Claim:} $\partial^{d-1}$ is surjective.

{\it Proof:}  As $\psi$ is injective we have from (\ref{eq:exact1}) that $\phi^d=0$, and thus $\phi^{d-1}=\Hom_R(S,\partial^{d-1})$ is surjective.   Let $C=\operatorname{coker} \partial^{d-1}$.   Then 
$$0\to C^{\vv}\to (J^d)^{\vv}\to \cdots \to (J^0)^{\vv}\to M^{\vv}\to 0$$
is exact.   Note that $(J^i)^{\vv}$ is a flat $R$-module for all $i$ (e.g., Corollary \ref{id-fd}(b)).  As the set of associated primes of any flat $R$-module is contained in the set of associated primes of $R$, and as $R$ is Cohen-Macaulay of dimension greater than zero, to show $C^{\vv}=0$ it suffices to show $(C^{\vv})_{\fp}=0$ for all $\fp\neq \fm$.  So fix a prime $\fp\neq \fm$. As $S$ is a finitely generated $R$-module, we have 
$\Tor_i^R(S,M^{\vv})\cong \Ext^i_R(S,M)^{\vv}=0$ for $i=1,\dots,d$ by Lemma \ref{duality-isos}(b).  This implies $\Tor_i^{R_{\fp}}(S_{\fp}, (M^{\vv})_{\fp})=0$ for $i=1,\dots,d$.  As $R_{\fp}$ is an $F$-finite Cohen-Macaulay local ring of dimension less than $d$, and $p^e\ge e(R)\ge e(R_{\fp})$, we have that $\fd_{R_{\fp}}(M^{\vv})_{\fp}<\infty$  by the induction hypothesis on part (b).  In particular, by \cite[Corollary 5.3]{B},  $\fd_{R_{\fp}} (M^{\vv})_{\fp}\le \dim R_{\fp}\le d-1$ and thus $(C^{\vv})_{\fp}$ is a flat $R_{\fp}$-module.   Then by either \cite[Corollary 3.5]{MW} or \cite[Theorem 3.1]{DIM}, we have 
\begin{equation}
\label{eq:exact3}0\to S_{\fp} \otimes_{R_{\fp}} (C^{\vv})_{\fp}\to S_{\fp}\otimes_{R_{\fp}} ((J^{d})^{\vv})_{\fp}\to S_{\fp}\otimes_{R_{\fp}} ((J^{d-1})^{\vv})_{\fp}
\end{equation}
is exact.  Now, since $\Hom_R(S,\partial^{d-1})$ is surjective, we have using duality and Lemma \ref{duality-isos}(b) that 
$$0\to S\otimes_R (J^{d})^{\vv}\to S\otimes_R (J^{d-1})^{\vv}$$ 
is exact.  Localizing this exact sequence at $\fp$ and comparing with (\ref{eq:exact3}), we have $S_{\fp} \otimes_{R_{\fp}} (C^{\vv})_{\fp}=0$.
However, tensoring with $S_{\fp}$ over $R_{\fp}$ is faithful (e.g., \cite[Propostion 2.1(c)]{M}) and hence $(C^{\vv})_{\fp}=0$.  Hence, $C^{\vv}=0$, and thus $C=0$, which completes the proof of the Claim.
\end{proof}

As a corollary, we obtain the equivalence of conditions (a) and (c) of Theorem \ref{main-theorem}:

\begin{corollary} \label{CM-flat}  Let $(R,\fm)$ be a $d$-dimensional Cohen-Macaulay local ring of prime characteristic $p$ and $M$ an $R$-complex such that $\hh_*(M)$ is bounded above.  Suppose there exist integers $e\ge \log_p e(R)$ and $t>\sup \hh_*(M)$ such that $\Tor_i^R(\ee R, M)=0$ for $t\le i\le t+r-1$, where $r=\max\{1, d\}$. Then $M$ has finite flat dimension.
\end{corollary}
\begin{proof}  By \cite[Section 3]{Ku} there exists a faithfully flat extension $S$ of $R$ such that $S$ is a $d$-dimensional CM local ring with an algebraically closed residue field and $e(S)=e(R)$.
Furthermore, by Corollary \ref{Frob-flat-ext}, $\Tor_i^S(\ee S, M\otimes_R S)=0$ for $t\le i\le t+r-1$.  Hence, by replacing $R$ with $S$ and $M$ with $M\otimes_R S$, we may assume $R$ is $F$-finite.   The result now follows from part (b) of Theorem \ref{inj-dim}. \end{proof}

\section{The general case}

We begin this section by proving a basic result concerning $E=E_R(k)$, the injective hull of the residue field of a local ring $(R,\fm, k)$.

\begin{lemma} \label{colon}  Let $(R,\fm,k)$ be a local ring. Then $$(0:_E (0:_R \fm))=\fm E.$$
\end{lemma}
\begin{proof} One containment is clear.  For the reverse inclusion, since $E\cong E_{\hat{R}}(\hat{R}/\hat{\fm})$,  $\hat{\fm} E=\fm E$ and $(0:_{\hat R}\hat{\fm})=(0:_R \fm)\hat{R}$, we may replace $R$ by $\hat R$ and assume $R$ is complete.  Consider the composition of maps

\begin{equation}
\label{comp1}
\Hom_R (R/(0:_R \fm), E)\cong (0:_E (0:_R \fm))\to E \to E/\fm E\cong E\otimes_R R/\fm.
\end{equation}

Dualizing, we have the composition
$$
(E\otimes_R R/\fm)^{\vv}\cong (0:_R \fm)\to R\to R/(0:_R\fm),
$$
which is clearly the zero map.  Thus, the composition (\ref{comp1}) is the zero map as well, implying $(0:_E (0:_R \fm))\subseteq \fm E$.

\end{proof}

We use the above lemma to prove the following:

\begin{lemma}
\label{image}
Let $(R,\fm,k)$ be a local ring and $\phi:J\to J'$ a homomorphism of injective $R$-modules.  Suppose $\Hom_R(R/\fm, J)\xrightarrow{\phi_*} \Hom_R(R/\fm, J')$ is zero.  Then $\phi(J)\subseteq \fm J'$.
\end{lemma}
\begin{proof} It suffices to prove the lemma in the case $J=E_R(R/\fp)$ and $J'=E_R(R/\fq)$, for $\fp, \fq\in \Spec R$,

\smallskip
\noindent
{\it Case 1:} $\fq\neq \fm$.

Then $\fm J'= \fm E_R(R/\fq) = \fm R_{\fq}\cdot E_R(R/\fq)=J'$, as $E_R(R/\fq)$ is an $R_{\fq}$-module.  So the lemma holds trivially.

\smallskip
\noindent
{\it Case 2:}  $\fq=\fm$ and $\fp\neq \fm$.  

Since $J=E_R(R/\fp)$ is an $R_{\fp}$-module, we have
$$(0:_R \fm)\phi(J)=\phi((0:_R \fm) R_{\fp}\cdot J)=\phi(0)=0.$$
Hence, $\phi(J)\subseteq (0:_{J'} (0:_R \fm))=\fm J'$ by Lemma \ref{colon}.

\smallskip
\noindent
{\it Case 3:} $\fp=\fq=\fm$.

In this case, $\phi$ is multiplication by some element $s\in \widehat{R}$.  If $s\not\in \widehat{\fm}$, then $\phi$ is an isomorphism, contradicting that $\Hom_R(R/\fm, \phi)$ is the zero map.  Thus, $s\in \widehat{\fm}$.  Hence, $\phi(J)\subseteq \widehat{\fm} J'= \fm J'$.

\end{proof}

\begin{lemma} \label{depth-zero}
Let $(R,\fm)$ be a local ring of depth zero and let $\ell$ be an integer such that $(0:_R \fm)\not\subset \fm^{\ell}$.   Let $J$ be an injective module such that $\mu_0(\fm, J)\neq 0$.  Then $(0:_J \fm^{\ell})\not\subset \fm J$.
\end{lemma}
\begin{proof}
It suffices to consider the case $J=E:=E_R(k)$.  Since the composition $(0:_R \fm)\to R\to R/\fm^{\ell}$ is nonzero,  the composition 
$$(R/\fm^{\ell})^{\vv}\cong (0:_E \fm^{\ell})\to E\to E/\fm E \cong \Hom_R(R/\fm, R)^{\vv}$$
is also nonzero.  Hence, $(0:_E \fm^{\ell})\not\subset \fm E$.
\end{proof}

\begin{lemma} \label{nonexact} Let $\phi:(R,\fm)\to (S,\fn)$ be a local homomorphism such that $S$ is a finitely generated $R$-module and $\depth S=0$.  Let $\ell$ be an integer such that $(0:_S \fn)\not\subseteq \fn^{\ell}$ and suppose $\fm S\subseteq \fn^{\ell}$.    Let $J^1\xrightarrow{\sigma} J^2 \xrightarrow{\tau} J^3$ be a sequence of maps of injective modules such that 
such that $\Hom_R(R/\fm, \sigma)=\Hom_R(R/\fm, \tau)=0$.   If  $\Hom_R(S, J^1)\xrightarrow{\sigma_*} \Hom_R (S, J^2)\xrightarrow{\tau_*} \Hom_R(S, J^3)$ is exact then $\mu_0(\fm, J^2)=0$.
\end{lemma}
\begin{proof}  Let $\widetilde{J^i}=\Hom_R(S,J^i)$ for $i=1,2,3$, which are injective $S$-modules.   Since $\fm S\subseteq \fn^{\ell}$, we have that $S/\fn^{\ell}\cong k^{r}$ as $R$-modules for some $r>0$, where $k=R/\fm$.   Consider the commutative diagram
$$
\begin{tikzcd} 
\Hom_S(S/\fn^{\ell}, \widetilde{J^1})\arrow [r, "\overline{\sigma_*}"]\arrow[d, "\cong"] & \Hom_S(S/\fn^{\ell}, \widetilde{J^2}) \arrow[d, "\cong"] \\
\Hom_R(S/\fn^{\ell}, J^1) \arrow[r] \arrow[d, "\cong"] & \Hom_R(S/\fn^{\ell}, J^2) \arrow[d, "\cong"] \\
\oplus\Hom_R(k, J^1) \arrow[r, "\oplus \overline{\sigma}"] &\oplus \Hom_R(k, J^2).
\end{tikzcd}
$$
As $\overline{\sigma}$ is the zero map by hypothesis, we see that $\overline{\sigma_*}$ is zero.   Similarly, the map $$\overline{\tau_*}:\Hom_S(S/\fn^{\ell}, \widetilde{J^2})\to \Hom_S(S/\fn^{\ell}, \widetilde{J^3})$$ is zero.
This implies that $(0:_{\widetilde{J^2}} \fn^{\ell})\subseteq \ker \tau^*$.   As $\overline{\sigma_*}$ is zero, we also have that the map $\Hom_S(S/\fn, \widetilde{J^1})\to \Hom_S(S/\fn, \widetilde{J^2})$ is zero.  By Lemma \ref{image}, this implies that $\im \sigma_* \subseteq \fn \widetilde{J^2}$.

Suppose $\mu_0(\fm, J^2)\neq 0$.  Since $\Hom_S(S, E_R(R/\fm))\cong E_S(S/\fn)$ by \cite[Lemma 3.7]{M}, we then have $\mu(\fn, \widetilde{J^2})\neq 0$.  By Lemma
\ref{depth-zero}, we have that $(0:_{\widetilde{J^2}} \fn^{\ell})\not\subset \fn \widetilde{J^2}$.   Hence, $\ker \tau_*\not\subset \im \sigma_*$, a contradiction.  Therefore, $\mu_0(\fm, J^2)=0$.

\end{proof}

\begin{theorem} \label{nonCM} Let $(R,\fm, k)$ be a $d$-dimensional local ring of prime characteristic $p$ which is $F$-finite.   Let $M$ be an $R$-complex such that $\hh^*(M)$ is bounded above.  Suppose there exists an integer $t>\sup \hh^*(M)$ such that for infinitely many integers $e$ one has  $\Ext^i_R(\ee R, M)=0$ for $t\le i\le t+r-1$, where $r=\max\{1,d\}$.   Then $M$ has finite injective dimension.
\end{theorem}
\begin{proof}  Precisely as in the initial paragraph of the proof of Theorem \ref{inj-dim},  we may assume $M$ is a module concentrated in degree 0 and $t=1$.  We proceed by induction on $d$, with the case $d=0$ being covered by Proposition \ref{zero-dim}.  Suppose now that $d\ge 1$ (so $r=d$) and let $\fp\neq \fm$ be a prime ideal.   Since $R$ is $F$-finite, $\Ext^i_{R_{\fp}}(\ee R_{\fp}, M_{\fp})=\Ext^i_R(\ee R, M)_{\fp}=0$ for infinitely many $e$ and $i=1,\dots, d$. Since $d\ge \max\{1, \dim R_\fp\}$, we have $\id_{R_\fp} M_{\fp}<\infty$ by the induction hypothesis.  Hence, $\id_{R_\fp} M_{\fp}\le \dim R_{\fp}\le d-1$.   Thus, $\mu_i(\fp, M)=0$ for all $i\ge d$ and all $\fp\neq \fm$.

If $R$ is Cohen-Macaulay we are done by Theorem \ref{inj-dim}.  Hence we may assume $s:=\depth R<d$ and it suffices to prove $\mu_d(\fm, M)=0$.  Let $e\ge 1$ be arbitrary and let $T$ denote the local ring $\ee R$ and $\fq$ the maximal ideal of $T$.  Let ${\bf x}=x_1,\dots,x_s\in \fq$ be a maximal regular sequence in $T$ and set $S:=T/({\bf x})$ and $\fn:=\fq S$.  Since $\depth S=0$,  there exists an integer $\ell$ (independent of $e$) such that $(0:_S \fn)\not\subset \fn^{\ell}$.   Now choose $e$ sufficiently large such that $p^e\ge \ell$ and $\Ext^i_R(T, M)=0$ for $i=1,\dots, d$.  Let
$$J:=\ 0\to J^0\to J^1\to J^2\to \cdots $$
be a minimal injective resolution of $M$, and for each $i$ let $\widetilde{J^i}$ denote $\Hom_R(T,J^i)$.
As $\Ext^i_R(T, M)=0$ for $1\le i\le d$, we see that
$$0\to \widetilde{J^0}\to \widetilde{J^1}\to \cdots \to \widetilde{J^{d}}\to \widetilde{J^{d+1}}$$
is part of an injective $T$-resolution of $\widetilde{M}:=\Hom_R(T,M)$.  Since $\pd_T S=s$ we have that $\Ext^i_T(S, \widetilde{M})=0$ for $i>s$.   In particular, as $d>s$ and $\Hom_T(S, \widetilde{J^i})\cong \Hom_R(S, J^i)$ for all $i$, we have that
$$\Hom_R(S, J^{d-1})\to \Hom_R(S, J^{d}) \to \Hom_R(S, J^{d+1})$$
is exact.   Since $\fm S\subseteq \fn^{[p^e]}\subseteq \fn^{\ell}$, we obtain that $\mu_d(\fm, M)=\mu_0(\fm, J^d)=0$ by Lemma \ref{nonexact}.   Thus, $\id_R M<\infty$.
\end{proof}

We now obtain the equivalence of conditions  (a) and (b) of Theorem \ref{main-theorem}:

\begin{corollary} \label{flat-nonCM}  Let $(R,\fm)$ be a $d$-dimensional local ring of prime characteristic $p$ and $M$ an $R$-complex such that $\hh_*(M)$ is bounded above.  Suppose there exist an integer $t>\sup \hh_*(M)$ such that  for infinitely many integers $e$ one has $\Tor_i^R(\ee R, M)=0$ for $t\le i\le t+r-1$, where $r=\max\{1,d\}$. Then $M$ has finite flat dimension.
\end{corollary}
\begin{proof} The argument is similar to the proof of Corollary \ref{CM-flat}, except one uses Theorem \ref{nonCM}  in place of Theorem \ref{inj-dim}.
\end{proof}

We close the paper by giving a proof that a theorem of Avramov and Miller \cite{AM} concerning finitely generated modules over complete intersections holds for arbitrary modules, and in fact any complex whose homology is bounded above.  The proof mostly follows the argument of Dutta \cite{D}, until the end when we apply \cite[Theorem 1.1]{DIM}.

\begin{theorem} \label{ci} Let $(R,\fm)$ be a local complete intersection ring of prime characteristic $p$.  Let $M$ be an $R$-complex such that $\hh_*(M)$ is bounded above.   Suppose $\Tor_i^R(\ee R,M)=0$ for some $e>0$ and some $i>\sup \hh_*(M)$.  Then $M$ has finite flat dimension.
\end{theorem}
\begin{proof}  Suppose $\Tor_i^R(\ee R,M)=0$ for some $e>0$ and some $i>\sup \hh_*(M)$.  A routine check shows that Steps 0-3 of the proof given in \cite{D} remain valid for arbitrary bounded above complexes.   Steps 2 and 3 yield that  $\Tor_{i+1}^R(\ee R, M)=0$ and $\Tor_i^R(\ef R, M)=0$.   Iterating, we obtain that $\Tor_j^R(\ee R, M)=0$ for all $j\ge i$ and all sufficiently large $e$.     By \cite[Theorem 1.1]{DIM}, $M$ has finite flat dimension.

\end{proof}

We deduce the dual version of the above result for complexes over $F$-finite local complete intersections:

\begin{corollary}  Let $(R,\fm)$ be a local complete intersection ring of prime characteristic $p$, and assume $R$ is $F$-finite.   Let $M$ be an $R$-complex such that $\hh^*(M)$ is bounded above.    Suppose $\Ext_i^R(\ee R,M)=0$ for some $e>0$ and some $i>\sup \hh^*(M)$.  Then $M$ has finite injective dimension.
\end{corollary}
\begin{proof}  By the argument in the initial paragraph of Theorem \ref{inj-dim}, we may assume $M$ is a module concentrated in degree zero.  
As $R$ is $F$-finite, we have by Lemma \ref{duality-isos} that $\Tor_i^{R}(\ee R, M^{\vv})=0$ for  some positive integers $i$ and $e$, where $(-)^{\vv}$ denotes
the functor $\Hom_{R}(-, E_{R}(R/\fm))$.   By Theorem \ref{ci}, we have $\fd_R M^{\vv}<\infty$. Hence, by Lemma \ref{id-fd}, $\id_R M=\fd_R M^{\vv}<\infty$.   
\end{proof}

\end{document}